\documentclass[10pt, reqno]{amsart}
\usepackage{fullpage, amsfonts,amsmath,amscd,amssymb}

\input xy
\xyoption{all}

\newtheorem*{lemma}{Lemma}

\newtheorem*{theorem}{Theorem}
\newtheorem*{corollary}{Corollary}

{

}

\theoremstyle{definition}

\newtheorem*{hypothesis}{Hypothesis}

\theoremstyle{remark}

\newcommand{\irr}[1]{\textsf{Irrep}(#1)}
\newcommand{\triv}{\textsf{triv}}

\def\CC{{\mathbb C}}

\def\fh{{V}}
\def\fhr{{\fh^{\text{reg}}}}

\DeclareMathOperator{\Hom}{Hom}

 \DeclareMathOperator{\id}{Id}
 
 \DeclareMathOperator{\md}{-mod}

\newcommand{\Z}{\mathbb{Z}}
\newcommand{\C}{\mathbb{C}}

\newcommand{\la}{\langle}
\newcommand{\ra}{\rangle}
\newcommand{\HH}{\mathbb{H}}
\newcommand{\OO}{\mathcal{O}}
\newcommand{\co}{\mathsf{co}}
\DeclareMathOperator{\Perm}{Perm}
\DeclareMathOperator{\GL}{GL}

\begin{document}

\pagenumbering{arabic}

\title{Catalan Numbers for Complex Reflection Groups}

\author{Iain Gordon and Stephen Griffeth}

\begin{abstract}We construct $(q,t)$-Catalan polynomials and
$q$-Fuss-Catalan polynomials for any irreducible complex reflection
group $W$. The two main ingredients in this construction are
Rouquier's  formulation of shift functors for the rational Cherednik
algebras of $W$, and Opdam's analysis of permutations of the
irreducible representations of $W$ arising from the
Knizhnik-Zamolodchikov connection. \end{abstract}

\address{School of Mathematics and Maxwell Institute of Mathematics, University of Edinburgh, Edinburgh, EH9 3JZ} \email{igordon@ed.ac.uk, sgriffeth@ed.ac.uk}

\thanks{We thank Maria Chlouveraki for many enlightening and patient
explanations of cyclotomic Hecke algebras and GAP, and also Pavel
Etingof, J\"urgen M\"uller and Vic Reiner for helpful input. Both authors are grateful for the full financial support of EPSRC grant EP/G007632}
\maketitle

\section{Introduction}
\subsection{Complex reflection groups}
Let $\fh$ be a complex vector space of dimension $n$. A {\it pseudo-reflection} is a non-trivial element $r$ of $\GL(\fh)$ that acts trivially on a hyperplane, called the reflecting hyperplane of $r$. Let $W$ be a finite subgroup of $\GL(\fh)$ generated by pseudo-reflections. The pair $(\fh, W)$ is call a {\it complex reflection group} and ${\fh}$ is called the {\it reflection representation} of $W$. We assume that ${\fh}$ is irreducible as a representation of $W$.

\subsection{} Denote by $\mathcal{A}$ the set of reflecting hyperplanes of $(\fh , W)$ and set $N := | \mathcal{A} |$. Similarly, denote by $\mathcal{R}$ the set of pseudo-reflections of $(\fh , W$) and set $N^\ast:= |\mathcal{R}|$.

\subsection{}\label{c-fn} Let $z = \sum_{r\in \mathcal{R}} (1-r)$, a central element of $\mathbb{C} [W]$. For any $U\in \irr{W}$, we set $c_U$ to be the integer by which $z$ acts on $U$. We define the {\it generalised Coxeter number} $h$ to be the integer $c_{\fh}$. An elementary calculation shows that \begin{equation} \label{coxeternumber} h = \frac{N+N^\ast}{n}.\end{equation} 

\subsection{Invariant theory} Let $P$ denote ring of polynomial functions on $\fh$. This carries a homogeneous action of $W$ and we set $(P_+^W)$ to be ideal of $P$ generated by $W$-invariant polynomials with zero constant term. The coinvariant algebra $P^{\co W} := P/(P_+^W)$ carries the regular representation of $W$. Given $U\in\irr{W}$, the {\it exponents} of $U$ $$e_1(U)\leq \ldots \leq  e_{\dim U}(U)$$ are the homogeneous embedding degrees of $U$ in $P^{\co W}$. These may be recorded in the {\it fake degree} of $U$ $$f_U(q) = \sum_{i=1}^{\dim U} q^{e_i(U)}.$$ Set $d_i = e_i(\fh) +1$ for $i=1, \ldots , n$: these are the {\it degrees} of a minimal set of homogeneous elements generating $P^W$.

\subsection{}
There is a permutation $\Psi \in \Perm (\irr{W})$ such that the fake degrees have a palindromic property
\begin{equation}\label{palin}f_U(q) = q^{c_U} f_{\Psi(U^*)} (q^{-1}).\end{equation}
For complex reflection groups this was first observed by Malle, \cite[Section 6B]{Mal}, and then explained in a case-free manner by Opdam, \cite[Proposition 7.4]{Opd}. 

\subsection{$q$-Fuss-Catalan numbers} \label{fuss}
For any positive integer $i$, set $[i]_q := 1+ q+ \cdots + q^{i-1}$. For a non-negative integer $m$, we define the $m$th {\it $q$-Fuss-Catalan number} of $(\fh , W)$ to be \begin{equation} \label{thenumbers} C^{(m)}_W(q) = \prod_{i=1}^n \frac{[mh+1+e_i(\Psi^m(\fh^*)^*)]_q}{[d_i]_q}.\end{equation}

\begin{theorem} \label{mainthm}  The rational function $C_W^{(m)}(q)$ belongs to $\mathbb{N}[q]$. Assuming Hypothesis~\ref{hecke} holds, two reasons for this are: \begin{enumerate} \item $C_W^{(m)}(q)$ is the Hilbert series of $(P/(\Theta))^W$ where $\Theta$ is a homogeneous system of parameters of degree $mh+1$ carrying the $W$-representation $\Psi^m(\fh^*)$;
\item $C_W^{(m)}(q)$ is the graded character of the finite dimensional irreducible representation $eL_{m+1/h}({\sf triv})$ of the spherical rational Cherednik algebra $U_{m+1/h }(W)$.
\end{enumerate}
\end{theorem}

\subsection{$(q,t)$-Catalan numbers} \label{q,t}
The description of $C_W^{(1)}(q)$ in the second part of the theorem allows us to define a $(q,t)$-Catalan number for all $W$ as follows. The rational Cherednik algebra representation $L_{1+1/h}({\sf triv})$ contains a unique copy of $\wedge^n \fh^*\in \irr{W}$: using this element to generate $L_{1+1/h}({\sf triv})$  one can then construct a filtration whose associated graded module carries the $(q,t)$-Catalan number by definition. By \cite[Theorem 5.11]{GoSt} this agrees with the definition of Garsia-Haiman in the symmetric group case, and should agree with the conjectural construction in \cite{stump}.

\subsection{Cyclic sieving phenomena} Let $d$ be a {\it regular} number for $(V,W)$ and let $\zeta = \exp(2\pi \sqrt{-1}/d)$, see for example \cite[2.2]{BeRe}. As pointed out in \cite[(5.1)]{OrSo}, for any $U\in \irr{W}$ there exists a permutation $\sigma \in \mathfrak{S}_n$ such that $$e_i(U) + e_{\sigma (i)}(U^*) \equiv 0 \quad \text{mod }d.$$ Combining this with the duality encoded in \eqref{palin} one shows by induction on $m$ that there exists a permutation $\rho\in\mathfrak{S}_n$ such that for all $i$ $$d_{\rho(i)} \equiv mh+1+ e_i(\Psi^m(V^*)^*) \quad \text{mod }d.$$ It follows from \cite[Proposition 3.2]{BeRe} that $C_W^{(m)}(\zeta^t)$ is a positive rational number for any $t\in \mathbb{Z}$; by the theorem it is also an algebraic integer. Hence $C_W^{(m)}(\zeta^t)$ is a positive integer for all $m$ and all $t$, and therefore a candidate for a cyclic sieving phenomena.

\subsection{Well-generated case}  The pair $(\fh , W)$ is {\it
well-generated} if $W$ can be generated by $n$ pseudo-reflections. It
is observed in \cite[Section 5]{OrSo} that in this case $h=d_n$. It
can be shown using \eqref{palin} and \cite[Proposition 5.2]{OrSo} that $e_i(V) + e_{n-i}(V^*)= h = e_i(\Psi^m(V^*)^*) + e_{n-i}(V^*)$, so that in this case the formula for the $q$-Fuss-Catalan number simplifies, $$C^{(m)}_W(q) = \prod_{i=1}^n \frac{[mh+d_i]_q}{[d_i]_q}.$$ This is the standard definition of $q$-Fuss-Catalan numbers for well-generated groups which is used throughout the literature. 

\subsection{}
In fact, a case-by-case observation made by Malle \cite[Corollary
4.9]{Mal} shows that $\Psi (\fh^\ast) = \fh^\ast$ if and only if
$(\fh, W)$ is well-generated. Thus, the first part of the theorem
above confirms \cite[Conjecture 4.3(i)]{BeRe} for those $W$ for which
Hypothesis~\ref{hecke} holds.

\subsection{Galois twists} We prove an analogue of the theorem above
for rational Cherednik algebras at any parameter $p/h$ where $p$ is a
positive integer coprime to $h$. The formulation of this theorem uses
certain twists of $\fh$ by Galois automorphisms of $\C$, as well as the permutation $\Psi$ of $\irr{W}$. See Theorem \ref{theorem1} for the precise statement. 

\subsection{}
In particular, for well-generated groups the graded character of $eL_{p/h}({\sf triv})$ is
\begin{equation} \label{twistcat}\prod_{i=1}^n \frac{[p+e_i({}^g V)]_q}{[d_i]_q},\end{equation}
where $g$ is an automorphism of $\C$ which maps $e^{2 \pi \sqrt{-1}/h}$ to $e^{2 \pi \sqrt{-1}p/h}$. This generalises the formula for the symmetric group $\mathfrak{S}_n$ $$\frac{1}{[n+p]_q} \left[ \begin{matrix} n+p \\ n \end{matrix}
\right]_q.$$

Moreover, if $p = mh-1$ then \eqref{twistcat} confirms
\cite[Conjecture 4.3(ii)]{BeRe} on ``positive" $q$-Fuss-Catalan
numbers for those $W$ for which Hypothesis~\ref{hecke} holds. 


\subsection{Layout of the paper} We give the proof of Theorem \ref{mainthm} and its Galois twist version in the next section, together with details for the others results mentioned here. Our key tools are the equivalences of highest weight categories discovered by Rouquier in \cite{Rou} and Opdam's study of the monodromy of the Knizhnik-Zamalodchikov connection in \cite{Opd}.  In the third section we give  data for exponents of $\Psi^m(\fh^*)^*$ in the case where $(\fh, W)$ is not well-generated, thus giving explicit formulae for the associated $q$-Fuss-Catalan numbers. These data were gathered with the help of the Chevie program in GAP.

\section{Proofs}

\subsection{Rational Cherednik algebras}
Let $R=\C[[{\bf k}]]$ be the ring of
formal power series in the indeterminate ${\bf k}$ and $Q = \C(( {\bf k}))$. For any $\C$-vector space $M$, we write $M_R$ for the extension $R\otimes_{\C} M$. Let $S$ be the ring $\textsf{Sym}(V)$ of symmetric functions on $V$. Let $k$ be a rational number which we will call the {\it parameter}. 

\subsection{} The 
\emph{rational Cherednik algebra} $\HH_{R, {k}}(\fh ,W)$ is the quotient of the $R$-algebra $T(\fh \oplus \fh^*)_R \rtimes W$ by relations $xy=yx$ if 
$x,y \in \fh$ or $x,y \in \fh^*$, and
\begin{equation}
yx-xy=\la x,y \ra+ ({\bf k} + k)\sum_{H \in \mathcal{A}} \frac{\la x,\alpha_H^\vee
\ra \la \alpha_H, y \ra}{\la \alpha_H, \alpha_H^\vee \ra} \sum_{w\in W_H} (1-\det(w)^{-1})w
\end{equation} if $x \in \fh^*$ and $y \in \fh$. Both $S_{R}$ and $P_{R}$ are subalgebras of $\HH_{R, k}(\fh , W)$. Here and throughout we will drop as many parts of the notation as we can: for instance we will write $\HH_{R}$ if both $k$ and the pair $(\fh , W)$ are clear from the context. We write $\HH_{\C, k}(\fh , W)$ or $\HH_{\C}$ for the specialisation of $\HH_{R, k}$ to $\C$ and similarly $\HH_{Q, k}(\fh , W)$ or $\HH_Q$. 

\subsection{Category $\mathcal{O}$}
\emph{Category $\mathcal{O}_{R, k}$} is the full subcategory of
finitely generated $\HH_{R}$-modules consisting of on which the operators in $V\subset S_{R} \subset \HH_{R}$ act locally nilpotently.  It is a highest weight category, \cite[Definition 4.11]{Rou}, 
with standard objects $\Delta_{R,k}(U) :=\text{Ind}^{\HH_{R}}_{S_{R} \rtimes W} U_{R}$ labelled by $U\in \irr{W}$ and partial order $U <_{k} U'$ if 
$k (c_{U'}-c_{U}) \in \mathbb{Z}_{>0}$ in the notation of \ref{c-fn}. We let $\mathcal{O}_{R, k}^{\Delta}$ denote the full subcategory of $\mathcal{O}_{R, k}$ whose objects admit a filtration by standard objects.
  
There are analogous definitions for $\mathcal{O}_{\C,k}$ and $\mathcal{O}_{Q,k}$, and base-change functors from $\mathcal{O}_{\bf k}$ to both categories.  

\subsection{Hecke algebras and the KZ functor} 
\label{hecke}
For each hyperplane $H \in \mathcal{A}$ let $e_H$ be the order
of the subgroup $W_H$ of $W$ that fixes $H$
pointwise.  Let $\fhr = \fh \setminus \bigcup_{H\in \mathcal{A}} H$, choose $x_0\in \fhr$, and let $B_W = \pi_1 (\fhr/W , x_0) $, the braid group of $W$. Let $\mathbf{H}_{R, k}$ be the 
\emph{Hecke algebra} of $W$ over $R$, \cite[Section 4.C]{BrMaRo}, the quotient of the group 
algebra $R[B_W]$ by relations 
\begin{equation} 
\label{Heckerelation}
(T_H - e^{2\pi i ({\bf k} + k)})\prod_{j=1}^{e_H-1} (T_H - \zeta_H^j) = 0 
\end{equation} where $T_H$ is a set of generators for $B_W$ running 
over a minimal set of reflecting hyperplanes.  

\begin{hypothesis} 
The algebra ${\bf H}_{R, k}$ is free over $R$ of rank $|W|$.
\end{hypothesis}

This is currently known to hold for all real $W$, all $W$ in the
infinite family $G(r,p,n)$, and most of the exceptional groups; see \cite{MaMi} for a recent report.

\subsection{} Following \cite[Section 5]{GGOR} there is an exact functor $${\sf KZ}_{R, k} : \mathcal{O}_{R, k} \longrightarrow {\bf H}_{R, k}\md $$ such that for $M,N \in \mathcal{O}_{R,k}^{\Delta}$ the natural map $\Hom_{\mathcal{O}_{R}} (M,N) \longrightarrow \Hom_{{\bf H}_{R}}({\sf KZ}_{R}(M) , {\sf KZ}_{R}(N))$ is an isomorphism. Similarly, there are functors ${\sf KZ}_{Q}: \mathcal{O}_{Q} \longrightarrow {\bf H}_{Q}\md$ and ${\sf KZ}_k: \mathcal{O}_{k} \longrightarrow {\bf H}_{k}\md$. The first is even an equivalence of categories, \cite[Corollary 2.20 and Theorem 5.14]{GGOR}, and so gives a bijection $$\tau_k: \irr{W} = \irr {\mathcal{O}_{Q,k}} \stackrel{\sim}\longrightarrow \irr {{\bf H}_{Q,k}}.$$

If $e^{2 \pi i k}\neq \zeta_H^j$ for all hyperplanes $H$ and integers $1\leq j \leq e_H-1$, then ${\sf KZ}_{R, k}$ is ``$1$-faithful" in the sense that if $M,N \in \mathcal{O}_{R,k}^{\Delta}$ then  $\text{Ext}^1_{\OO_R}(M,N) \longrightarrow \text{Ext}^1_{{\bf H}_R}({\sf KZ}_R(M),{\sf KZ}_R(N))$ is an isomorphism , \cite[Theorem 5.3]{Rou}.

\subsection{Equivalences} \label{equivalences}

Let $g\in\text{Aut}(\C/\mathbb{Q})$ and  
let $k$ and $k'$ be parameters such that 
\begin{equation} \label{g related}
g(e^{2 \pi i k})=e^{2 \pi i k'}.
\end{equation}  Then $g$ extends to an automorphism of $R$ which fixes 
$\mathbf{k}$, and the isomorphism $\gamma: R[B_W] \longrightarrow R[B_W]^g$ defined by $\gamma(\sum c_b b) = \sum g(c_b) b$ 
 descends to an isomorphism of ${\bf H}_{R,k}$ onto 
${\bf H}_{R, k'}^g$.  Let $\mathcal{O}_{R,k'}^g$ denote the category whose objects and morphisms are the same as those for $\mathcal{O}_{R,k'}$ as sets, but such that the $R$-linear structure on morphisms is twisted by $g$. Then ${\sf KZ}_{R,k'}$ induces an $R$-linear functor from $\mathcal{O}_{R,k}^g$ to ${\bf H}_{R,k}^g\md$.
Changing base to $K$ this defines a permutation $\phi_{k,k'}^g\in\text{Perm}(\irr{W})$ via the isomorphism in ${\bf H}_{Q,k'}^g\md$ \begin{equation} \label{perm} {}^{\gamma}{\sf KZ}_{Q,k}(\Delta_{Q,k}(U)) \cong {\sf KZ}_{Q,k'}(\Delta_{Q,k'}(\phi_{k,k'}^g(U)))\end{equation}

\subsection{} The following theorem is at the heart of our work.
\begin{theorem}[{\cite[Theorems 4.49 and 5.5]{Rou}}] \label{equiv theorem}
Keep the above notation and assume that ${\sf KZ}_{R,k}$ and ${\sf KZ}_{R,k'}$ are both $1$-faithful.  If $U <_{k} U'$ if and only if $\phi_{k,k'}^g(U)<_{k'} \phi_{k,k'}^g(U')$, then there is an equivalence 
$S_{k,k'}:\OO_{R,k} \stackrel{\sim}\longrightarrow \OO_{R,k'}^g$ 
of highest weight categories such that 
$S_{k,k'}(\Delta_{R,k}(U))=\Delta_{R,k'}(\phi_{k,k'}^g(U))$ for all $U\in\irr{W}$.
\end{theorem}

The equivalence of categories in this theorem can be specialised to produce an equivalence $\mathcal{O}_{\C, k}\longrightarrow \mathcal{O}_{\C,k'}^g$ with the same properties as above. 

\subsection{Local data} \label{local} By construction, the permutation $\phi_{k,k'}^g$ preserves the dimension of a representation $U\in\irr{W}$. It is an important result of Opdam that $\phi_{k,k'}^g$ also respects the {\it local data} of $U$, that is the set of integers $\{ n_{H,j}^U \}$ defined by $$\text{Res}_{W_H}^W U \cong \bigoplus_{0\leq j\leq e_H-1} n_{H,j}^U \text{det}^{-j}.$$ To be explicit, \cite[(3.8)]{Opd} shows that the action of $T_H$ on ${\sf KZ}_{K,k}(\Delta_{K,k}(U))$ diagonalises to $$M_H(k) = \text{diag}( \zeta_H^0 e^{2\pi i {\bf k} + k}\id_{n_{H,0}^U}, \zeta_H^1 e^{2\pi i {\bf k}}\id_{n_{H,1}^U}, \cdots , \zeta_H^{e_H-1} e^{2\pi i {\bf k}}\id_{n_{H,e_H-1}^U})$$ It follows then from the definition in \eqref{perm} that $n_{H,j}^{\phi_{k,k'}^g(U)} = n_{H,j}^{{}^{g}U}$ where ${}^gU$ is the representation of $W$ obtained by applying $g$ to the entries in a matrix representation of $U$. 

There are two useful consequences. First, $\phi_{k,k'}^g(\triv) = \triv$ for any automorphism $g$. Second, recall the element $z\in \mathbb{Z}[W]$ introduced in \ref{c-fn}. It may be written as $z = N+N^{\ast} - \sum_{H\in \mathcal{A}} \sum_{w\in W_H}w$ and from this it follows that $c_U = N+N^{\ast} - (\dim U)^{-1}\sum_{H\in\mathcal{A}} e_Hn_{H,0}^U.$ Since twisting by $g$ does not change the dimension of the trivial eigenspace, we have $n_{H,0}^{\phi_{k,k'}^g(U)} = n_{H,0}^{{}^gU} = n_{H,0}^U$. We deduce that $c_{\phi_{k,k'}^g(U)}=c_U
$ for all $U\in \irr{W}$.

\subsection{}
We can now check the hypothesis of the above theorem in the case we will require. 
\begin{corollary} \label{Galois equiv}
Let $g$ be an automorphism of $R$ as in \ref{equivalences}, with $g(e^{2\pi i k}) = e^{2\pi ir k}$ where $r\in \mathbb{R}_{> 0}$.
If ${\sf KZ_{R,k}}$ is $1$-faithful, then there is an equivalence 
 $\OO_{\C,k} \longrightarrow \OO_{\C,rk}^g$ of highest weight covers of $\mathcal{H}_{\C,k} \cong \mathcal{H}_{\C,rk}^g$, mapping $\Delta_{k}(U)$ to $\Delta_{rk}(\phi_{k,k'}^g(U))$ for all $U \in \irr{W}$.
\end{corollary}
\begin{proof}
Observe first that ${\sf KZ}_{R,rk}$ is $1$-faithful by \eqref{g
related} since ${\sf KZ}_{R,k}$ is $1$-faithful. By \eqref{g related}
both $e^{2 \pi i k}$ and $e^{2\pi i rk}$ are roots of unity of the
same order, so $k(c_{U'} - c_{U})\in \Z$ if and only if $rk
(c_{\phi_{k,k'}^g(U')}- c_{\phi_{k,k'}^g(U)} ) = rk(c_{U'} - c_{U})\in
\Z$. Since $r$ is positive, we deduce that $U<_k U'$ if and only if
$\phi_{k,k'}^g(U)<_{rk} \phi_{k,k'}^g(U')$. The corollary follows from
the statement following Theorem \ref{equiv theorem}. \end{proof}

\subsection{Catalan numbers}
Let $L_{k}(\triv)$ denote the simple quotient of $\Delta_k(\triv)$. The following lemma is proved by the same argument as \cite[Proposition 2.1]{BEG}.

\begin{lemma} \label{basic}  $L_{-\frac{1}{h}}(\triv) = \triv$.\end{lemma}

\subsection{} We now prove our main result, answering a question of the second author, \cite[Section 8]{Gri}, and giving a general and case-free construction of the Koszul resolutions produced in \cite{BEG, Gor, Gri, Val}.

\begin{theorem} \label{theorem1}
Let $r$ be a positive integer coprime to $h$ and suppose $g\in\text{Aut}(\C/\mathbb{Q})$ sends $e^{-2 \pi i /h}$ to $e^{-2 \pi i r/h}$. Then there is an exact sequence in $\mathcal{O}_{\C,-\frac{r}{h}}$
$$0\rightarrow \Delta_{-\frac{r}{h}}(\wedge^n\phi_{k,k'}^g(\fh^*)) \rightarrow \cdots \rightarrow \Delta_{-\frac{r}{h}}(\wedge^2\phi_{k,k'}^g(\fh^*)) \rightarrow \Delta_{-\frac{r}{h}}(\phi_{k,k'}^g(\fh^*)) \rightarrow \Delta_{-\frac{r}{h}}(\triv) \rightarrow L_{-\frac{r}{h}}(\triv) \rightarrow 0$$
with $L_{-\frac{r}{h}}(\triv)$ finite dimensional.  \end{theorem}
\begin{proof}
The rank one case is easy and left to the reader; we assume that the rank is at least two. In case $k=-1/h$ we have by Lemma \ref{basic} an exact sequence $\Delta_{-\frac{1}{h}}(\fh^*) \longrightarrow \Delta_{-\frac{1}{h}}(\text{triv}) \longrightarrow L_{-\frac{1}{h}}(\triv) \rightarrow 0.$ Since $L_{-\frac{1}{h}}(\triv)$ is finite dimensional, it is elementary that this extends to a resolution in $\mathcal{O}_{\C,-\frac{1}{h}}$ \begin{equation} \label{resol} 0\rightarrow \Delta_{-\frac{1}{h}}(\wedge^n\fh^*) \rightarrow \cdots \rightarrow \Delta_{-\frac{1}{h}}(\wedge^2\fh^*)\rightarrow \Delta_{-\frac{1}{h}}(\fh^*) \rightarrow \Delta_{-\frac{1}{h}}(\triv) \rightarrow L_{-\frac{1}{h}}(\triv) \rightarrow 0,\end{equation} see \cite[Lemma 3.1]{Gri}.

We wish to apply Corollary~\ref{Galois equiv} to this resolution, so we need to know that for all complex reflection groups of rank 
at least $2$ we have inequalities $e^{-2 \pi i/h} \neq \zeta_H^j$ 
for all hyperplanes $H$ and $1\leq j\leq e_H-1$. To see this, observe first that $
nh=\sum_{H \in \mathcal{A}} e_H$. Since $W$ is acting irreducibly on $\fh$ there are at least $n$ summands on the right hand side 
accounted for by a $W$-orbit of hyperplanes $H$ maximizing $e_H$. Thus 
$h \geq e_H$ for all $H$. If we suppose that $h=e_H$ for some $H$, then we must have $\mathcal{A} = W\cdot H$ and $N=n$. Choosing linear 
forms defining the hyperplanes gives a basis of $\fh^*$, and restricting any invariant polynomial to these hyperplanes shows that the degrees, $d_i$, of all the homogeneous generators of $P^W$ have degree divisible by $e_H$. By Molien's theorem, however, $\sum d_i = N^* + n = ne_H$ and so $d_i = e_H$ for each $i$. This implies that $W$ is a product of $n$ cyclic groups, but since there was assumed to be only one orbit of hyperplanes, we deduce that $n=1$. Thus, since we assume the rank is greater than $1$, we have $h>e_H$ for all $H$ and that implies the required inequalities. 

Applying the equivalence of Corollary~\ref{Galois equiv} to \eqref{resol} produces an exact sequence in $\mathcal{O}_{-\frac{r}{h}}^g$ 
\begin{eqnarray} \label{resol1} 0\rightarrow \Delta_{-\frac{r}{h}}(\phi_{k,k'}^g(\wedge^n\fh^*)) \rightarrow \cdots &\rightarrow &\Delta_{-\frac{r}{h}}(\phi_{k,k'}^g(\wedge^2 \fh^*)) \rightarrow \Delta_{-\frac{r}{h}}(\phi_{k,k'}^g(\fh^*)) \notag \\ && \rightarrow  \Delta_{-\frac{r}{h}}(\phi_{k,k'}^g(\triv)) \rightarrow  L_{-\frac{r}{h}}(\phi_{k,k'}^g(\triv)) \rightarrow 0,\end{eqnarray} 
By \cite[Corollary 4.14]{GGOR} $L_{-\frac{r}{h}}(\phi_{k,k'}^g(\triv))$ is finite dimensional and by \ref{local} $\phi_{k,k'}^g(\triv) = \triv$.  It follows that the image of the generating weight space $\phi_{k,k'}^g(V^*)\subset \Delta_{-\frac{r}{h}}(\phi_{k,k'}^g(V^*))$ in $\Delta_{-\frac{r}{h}}(\triv)\cong P$ is the linear span of a regular sequence $\Theta$. Thus \eqref{resol1} is just a Koszul resolution when restricted to $P\rtimes W$ and it follows that the generating weight spaces $\phi_{k,k'}^g(\wedge^iV^*)\subset \Delta_{-\frac{r}{h}}(\phi_{k,k'}^g(\wedge^iV^*))$ must be isomorphic to $\wedge^i\phi_{k,k'}^g(V^*)$. Considering the sequence in $\mathcal{O}_{\C, -\frac{r}{h}}$ instead of $\mathcal{O}_{\C, -\frac{r}{h}}^g$ completes the proof of the theorem. 
\end{proof}

\subsection{} We apply this theorem first to the case $r = mh+1$ for some positive integer $m$. In this case we may take $g$ to be the identity. It is not true, however, that $\phi_{k,k'}^{\text{id}}$ is the trivial permutation of $\irr{W}$! This is a consequence of the fact that the ${\sf KZ}$ functor varies in the parameter ${\bf k}$ rather than in its exponential $e^{2\pi i {\bf k}}$: this phenomenon has been studied and applied by Opdam, \cite{Opd, Opd1}.

Let $\Psi = \phi_{-\frac{1}{h}, -1-\frac{1}{h}}^{\text{id}} \in
\text{Perm}(\irr{W})$, so that in this case $\Psi$ is the permutation
on $\irr{W}$ induced by the equivalence ${\sf KZ}^{-1}_{K,k-1}\circ
{\sf KZ}_{K,k} : \mathcal{O}_{K,k} \longrightarrow
\mathcal{O}_{K,k-1}$ applied at $k=-\frac{1}{h}$.  Then $\Psi^m$ equals $\phi_{-\frac{1}{h}, -m - \frac{1}{h}}^{\text{id}}$ for general $m$. It follows from \cite[Proposition 7.4]{Opd} that $\Psi$ satisfies \eqref{palin}.

\subsection{Proof of Theorem \ref{mainthm}}
Part (1) is now a straightforward application of the standard invariant theory arguments in \cite[Proposition 4.2]{BeRe}. To apply these we need to know the degree of the image of the generating set $\Theta \subset \Delta_{-m-\frac{1}{h}}{\Psi^m(V^*)}$ in $\Delta_{-m-\frac{1}{h}}(\triv) \cong P$ used in Theorem \ref{theorem1}: this is just $(m+\frac{1}{h})c_{\Psi^m(V^*)} = (m+\frac{1}{h})c_{V^*} = mh+1$. It now follows from \cite[(4.3)]{BeRe} that the graded $W$-character of $(P/(\Theta))^W$ is given by the formula for $C^{(m)}_W(q)$ given in \ref{fuss}. (Although the formula in \cite{BeRe} is stated only for Galois conjugates ${}^{\sigma}V$ of $V$, this hypothesis is used to know that $(P\otimes \wedge^{\bullet} ({}^{\sigma}V))$ is free over $P^W$; but by \cite[Theorem 3.1]{OrSo} freeness holds for all representations $U\in \irr{W}$ satisfying $\sum_{i} e_i(U) = e_1(\wedge^{\text{top}} U)$. This equality holds for $U = \Psi^m(\fh^*)$ since, by \cite[Lemma 2.1]{Opd}, $\sum_i e_i(U) = \sum_{H\in \mathcal{A}} \sum_{j=1}^{e_H-1} jn_{H,j}^U$ for any $U\in\irr{W}$, so by \ref{local} $\sum_i e_i(\Psi^m \fh^*) = \sum_i e_i(\fh^*) = e_1(\wedge^{\text{top}} \fh^*)   = e_1 (\wedge^{\text{top}} (\Psi^m \fh^*)).$)

Part (2) follows since for any pair of dual bases $\{x_i\}, \{y_i\}$ of $\fh^*$ and $\fh$ the element ${\bf h} = \frac{1}{2}\sum_{i=1}^n (x_iy_i + y_ix_i)\in \mathbb{H}_{\C,-m-\frac{1}{h}}$ acts as a grading operator where non-zero elements of $\fh^*\subset P$ have degree $1$, see for instance \cite[Section 3.1]{GGOR}. \hfill $\Box$

\subsection{(q,t)-Catalan numbers} To deduce the existence of $(q,t)$-Catalan numbers as in \ref{q,t} we need to show that $\wedge^n \fh^*$ appears with multiplicity one in $L_{-1-\frac{1}{h}}(\triv)$. Now a reflection $s_H$ acts on $\wedge^{\sf top}U$ by the scalar $\zeta_H^{- \sum_j n_{H,j}^U}$, so it follows from \ref{local} that $\wedge^n \fh^* \cong \wedge^n \Psi (\fh^*)$. The multiplicity one result is \cite[Theorem 3.2]{Gri} if we show that $e_i(\Psi (\fh^*)) + d_{n-i} = h+1$ for all $i$. But this is an immediate consequence of \eqref{palin}.  

\subsection{} We also remark that the above observation produces a $W$-stable quotient  of the diagonal coinvariant ring ${\sf Sym}(\fh \oplus \fh^*)^{{\sf co} W}$ with pleasant properties, see \cite[Theorem 3.2]{Gri}. 

\subsection{Calculating $\phi_{k,k'}^g$} \label{calculate} Let
$E=\CC[{\bf v}^{\pm 1}]$ be the ring of Laurent polynomials in the variable ${\bf v}$ and let $K= \C ( {\bf v})$. We may define a Hecke algebra ${\bf H}_{E}$ as a quotient of $E[B_W]$ by the relations \eqref{Heckerelation}, where we replace $e^{2\pi i ({\bf k}+k)}$ by ${\bf v}^{\ell}$ for some positive integer $\ell$. 
By \cite[Theorem 6.7]{Opd}, for $\ell$ large enough depending on $W$, $K$ is a splitting field for $\mathbf{H}_{E}$. (In fact by \cite[Corollary 4.8]{Mal} we may take $\ell$ to be number of roots of unity belonging to the field of definition of the $W$-representation $V$.)  Fix such an $\ell$ and let $\chi_U$ be the character of ${\sf KZ}_{K,k}(\Delta_{K,k}(U))$, so that $\chi_U (b) \in K$ for all $b\in B_W$. By definition we then have $$\chi_{\phi_{k,k'}^g(U)}(b)\vert_{{\bf v} =  e^{2 \pi i ({\bf k}+k')/\ell}}=g\left(\chi_{U}(b) \vert_{{\bf v} = e^{2 \pi i ({\bf k}+k)/\ell}} \right).$$ This may be rewritten as \begin{equation} \label{whatweuse}\chi_{\phi_{k,k'}^g(U)}(b) ({\bf v}) =(g \chi_{U}(b))(\eta{\bf v})\end{equation} where $\eta = e^{-2 \pi i k'/\ell} g \left(e^{2\pi i k/\ell}\right)$ is an $\ell$th root of unity and $g$ acts on $\chi_U(b) \in K$ fixing ${\bf v}$. This last formula gives an effective method for calculating $\phi^g_{k,k'}$ in examples. 

\subsection{} \label{dual} We can now also illustrate Theorem \ref{theorem1} by applying it in the case $k=-\frac{1}{h}$ and $r = mh-1$ with $g$ being complex conjugation. In this situation we have $\eta = e^{2\pi i m/\ell}$, so it follows from \eqref{whatweuse} that $\phi^g_{-\frac{1}{h}, -m+\frac{1}{h}} (U) = \phi_{-\frac{1}{h},-m-\frac{1}{h}}^{\text{id}}(U^*) = \Psi^m(U^*)$. Considerations similar to the Proof of Theorem \ref{mainthm} then show that in this case $(P/(\Theta))^W$ has Hilbert series $$\prod_{i=1}^n \frac{[mh-1+e_i(\Psi^m(V)^*)]_q}{[d_i]_q}.$$

\subsection{Well-generated case} \label{wgc} With the exception of confirming Hypothesis \ref{hecke}, our arguments so far have been entirely case-free. Arguing case-by-case, however, \cite[Corollary 4.9]{Mal} shows that $\C({\bf v}^{\ell})$ is a splitting field for the reflection representation of ${\bf H}_E$ and its dual precisely when $(V,W)$ is well-generated. It follows that $\chi_{\fh^*}(b)$ is a function in ${\bf v}^{\ell}$ and so \eqref{whatweuse} implies that $\phi^g_{k,k'}(\fh^*) = {}^g (\fh^*)$. Thus, in this case, Theorem \ref{mainthm} confirms \cite[Conjecture 4.3(i)]{BeRe} and Theorem \ref{theorem1} combined with the analysis in \ref{dual} confirms \cite[Conjecture 4.3(ii)]{BeRe}. 

\subsection{Remarks on several parameters} 
In general rational Cherednik algebras depend on parameters $\{ k_{H,j} : H\in \mathcal{A}/W, 0\leq j \leq e_H-1 \}$; we have considered only the case where $k_{H,0} = k$ and $k_{H,j} = 0$ for $j\neq 0$. Nevertheless, the techniques we use here to analyse the equivalences of \cite{Rou} extend to the general case. 

Explicitly, there is a permutation $\phi^g_{(k_{H,j}), (k'_{H,j})} \in \text{Perm}(\irr{W})$ attached to a potential shift from $\mathcal{O}_{\C, (k_{H,j})}$ to $\mathcal{O}_{\C, (k'_{H,j})}$ and to apply \cite[Theorems 4.49 and 5.5]{Rou} to deduce an equivalence, one must check the condition $U <_{(k_{H,j})} U'$ if and only if $\phi_{(k_{H,j}),(k'_{H,j})}^g(U)<_{(k'_{H,j})} \phi_{(k_{H,j}),(k'_{H,j})}^g(U')$. By definition, $U <_{(k_{H,j})} U'$ if $\sum_{H\in \mathcal{A}} \sum_{j=0}^{e_H-1} e_H k_{H,j}\left((\dim U)^{-1}n_{H,-j}^U - (\dim U')^{-1}n_{H,-j}^{U'}\right)\in \mathbb{Z}_{>0}$, so the condition can be calculated from the local data. As in \ref{local} we have then $n_{H,j}^{\phi^g_{(k_{H,j}), (k'_{H,j})}(U)} = n_{H,j}^{{}^gU}$ and so we must check $U <_{(k_{H,j})} U'$ if and only if ${}^gU<_{(k'_{H,j})}{}^gU'$.

In particular, if $(k'_{H,j})$ is obtained from $(k'_{H,j})$ by the addition of integers then we may take $g = \text{id}$ and we see that we can apply \cite[Theorem 5.5]{Rou} as stated; more generally, in \cite[Proposition 5.14]{Rou} we should use the formalism here. 

\section{Data}
Here we record the details necessary to calculate the Fuss-Catalan numbers defined by \eqref{thenumbers} for all irreducible complex reflection groups that are not well-generated. For the imprimitive groups $G(de,e,n)$ these data can be found in \cite[Section 8]{Gri}; for the exceptional groups we used \ref{calculate} combined with the detailed data of characters and Schur elements for Hecke algebras available in the Chevie program in GAP. (Note that degrees of $W$ can be read off from the exponents of $\fh = \Psi^0(\fh^*)^*$.)

\medskip
{\footnotesize
\begin{center}
\begin{tabular}{|l|l|l|l|l|}
\hline
Group & Generalised & Order of $\Psi$ & Exponents of $\Psi^m(\fh^*)^*$ \\
& Coxeter number & acting on $V^*$ &   \\
\hline \hline
$G(de,e,n)$  & $de(n-1)+d$ & $e$ & $d(e-m)-1, d(2e-m)-1 , \ldots ,  d((n-1)e-m) - 1, $  \\
$e>1, d>1$ & & & $d(m(n-1) + n) - 1$ ($0\leq m \leq e-1$) \\
\hline
$G_7$ & 18 & 2 & 11, 11 ($m=0$)  \\
& & & 5, 17 ($m=1$)  \\
\hline
$G_{11}$ & 36 & 2 & 23, 23 ($m=0$)  \\
& & & 11, 35 ($m=1$) \\
\hline
$G_{12}$ & 12 & 2 & 5, 7 ($m=0$) \\ & & & 1, 11 ($m=1$)  \\
\hline
$G_{13}$ & 18 & 4 & 7, 11 ($m=0$)  \\ &&& 5, 13 ($m=1$)  \\ &&& 7,11 ($m=2$)  \\ &&& 1,17 ($m=3$)  \\
\hline
$G_{15}$ & 30 & 4 & 11, 23 ($m=0$)  \\ &&& 17,17 ($m=1$)  \\ &&& 11,23 ($m=2$)  \\ &&& 5, 29 ($m=3$)  \\
\hline
$G_{19}$ & 90 & 2 & 59, 59 ($m=0$)  \\ &&& 29, 89 ($m=1$)  \\
\hline
$G_{22}$ & 30 & 2 & 11, 19 ($m=0$)  \\ &&& 1, 29 ($m=1$)  \\ 
\hline
$G_{31}$ & 30 & 2 & 7,11,19,23 ($m=0$)  \\ &&& 1,13,17,29 ($m=1$) \\
\hline
\end{tabular}
\end{center}}
\def\cprime{$'$} \def\cprime{$'$}

\end{document}